\documentclass[12pt]{amsart}
\usepackage{amsmath, amssymb,graphicx,caption, subcaption, mathtools, enumitem}
\usepackage[pdftex,hyperindex=true]{hyperref}
\usepackage{etoolbox}
\patchcmd{\section}{\scshape}{\bfseries}{}{}
\patchcmd{\subsection}{\bfseries}{\itshape}{}{}

\makeatletter
\def\@seccntformat#1{%
  \protect\textup{\protect\@secnumfont
    \ifnum\pdfstrcmp{section}{#1}=0 \bfseries\fi
    \ifnum\pdfstrcmp{subsection}{#1}=0 \itshape \fi
    \csname the#1\endcsname
    \protect\@secnumpunct
  }%
}  
\makeatother

\usepackage[utf8]{inputenc}
\usepackage[english]{babel}

\theoremstyle{plain}

\newtheorem{theorem}{Theorem}[section]

\newtheorem{corollary}[theorem]{Corollary}
\newtheorem{lemma}[theorem]{Lemma}

\theoremstyle{definition}
\newtheorem{definition}[theorem]{Definition}

\newtheoremstyle{note}
{3pt}
{3pt}
{\itshape}
{}
{\itshape}
{:}
{.5em}
{}

\theoremstyle{note}
\newtheorem{remark}[theorem]{Remark}

\newtheorem{example}[theorem]{Example}

\DeclareMathOperator{\Id}{Id}
\DeclareMathOperator{\Lie}{Lie}
\DeclareMathOperator{\ad}{ad}
\DeclareMathOperator{\Ad}{Ad}
\DeclareMathOperator{\Conv}{Conv}
\DeclareMathOperator{\Ker}{Ker}

\DeclareMathOperator{\diag}{diag}
\DeclareMathOperator{\Circ}{Circ}
\DeclareMathOperator{\rank}{rank}
\DeclareMathOperator{\Comm}{Comm}
\DeclareMathOperator{\Span}{Span}

\numberwithin{equation}{section}

\newcommand{\acknowledge}{\subsection*{Acknowledgements}}

\def\Dbar{\leavevmode\lower.6ex\hbox to 0pt{\hskip-.23ex \accent"16\hss}D}

\begin{document}

\title[Commutators and Cartan subalgebras]{Commutators and Cartan subalgebras in Lie algebras of compact semisimple Lie groups}
\author{J. Malkoun}
\address{Department of Mathematics and Statistics\\
Notre Dame University-Louaize, Zouk Mikael\\
Lebanon}
\email{joseph.malkoun@ndu.edu.lb}

\author{N. Nahlus}
\address{Department of Mathematics\\
Faculty of Arts and Sciences\\
American University of Beirut\\
Beirut, Lebanon\\}
\email{nahlus@aub.edu.lb}

\date{Received: date / Accepted: date}

\maketitle

\begin{abstract} First we give a new proof of Goto's theorem for Lie algebras of compact semisimple Lie groups 
using Coxeter transformations. Namely, every $x$ in $L = \Lie(G)$ can be written as $x =[a, b]$ for some $a$, $b$ 
in $L$. By using the same 
method, we  give a new proof of the following theorem (thus avoiding the classification tables of fundamental 
weights):
\emph{in compact semisimple Lie algebras, orthogonal Cartan subalgebras always exist} (where one of them can be 
chosen arbitrarily). Some of the 
consequences of this theorem are the following. $(i)$ If $L=\Lie(G)$ is such a Lie algebra and $C$ is any Cartan subalgebra of $L$, then the $G$-orbit of 
$C^{\perp}$ is all of $L$. $(ii)$ The consequence in part $(i)$ answers a question by L. Florit and W. Ziller on 
fatness of certain principal bundles. It also shows that in our case, the commutator map $L \times L \to L$ is 
open at $(0, 0)$. $(iii)$ given any regular element $x$ of $L$, there exists a regular element $y$ such that
$L = [x, L] + [y,L]$ and $x$, $y$ are orthogonal. Then we generalize this result about compact semisimple 
Lie algebras to the class of non-Hermitian real semisimple Lie algebras having full rank.

Finally, we survey some recent related results , and construct explicitly orthogonal Cartan subalgebras in 
$\mathfrak{su}(n)$, $\mathfrak{sp}(n)$, $\mathfrak{so}(n)$.
\end{abstract}

\maketitle

\section{Introduction} \label{section1}

The Lie algebra version of Goto's Theorem \cite{Got} for compact semisimple Lie groups is not as
well known as it ought to be. It says that
\begin{theorem} \label{thm-1_1} Let $L=\Lie(G)$ be the Lie algebra of a compact semisimple Lie group $G$.
Then every element $x$ in $L$, can be written as $x =[a, b]$ for some $a$, $b$ in $L$.
\end{theorem}
The first proof of Theorem \ref{thm-1_1} (i.e. Goto's Theorem, additive version) was proved by 
Karl-Hermann Neeb \cite[p. 653]{H-M2} (that he has communicated to the authors of \cite{H-M2}). His proof 
was based on Kostant's convexity Theorem \cite[Thm. 8.1]{Kos}, or more precisely, the version of Kostant's
convexity Theorem for compact semisimple Lie groups. Such proof will be presented in
Appendix \ref{sectionA} but slightly simplified using \cite[Lemma 2.2]{Akh}.
In section \ref{section2}, we prove Theorem \ref{thm-1_1} directly by using any Coxeter transformation of the
Weyl group $W(L, C)$ where $C$ is a maximal toral subalgebra of $L$. We remark that certain
Coxeter transformations were used to prove Goto's Theorem (on the group level) as in
Bourbaki's book \cite[Corollary, section 4 of chapter 9]{Bor2} or Hofmann and Morris's book \cite[Corollary 6.56]{H-M2}. 
To the best of our knowledge, this direct proof is new.
In section \ref{section3}, we prove the following theorem.
\begin{theorem} \label{thm-1_2} Let $\Lie(G)$ be the Lie algebra of a compact semisimple Lie group $G$.
Then, with respect to the negative of its Killing form, $\Lie(G)$ has orthogonal Cartan subalgebras 
(maximal toral subalgebras) where one of them can be chosen arbitrarily.\end{theorem}
This interesting result was proved recently by d’Andrea and Maffei in \cite[Lemma 2.2]{A-M}, via
the classification tables of fundamental weights for all types of simple Lie algebras
(except type $A_n$). Specifically, their proof uses the tables in \cite{Bor2} to verify that (in any root system of $(L, C)$ 
where $C$ is a Cartan subalgebra of $L$), the
highest root is either equal or twice some fundamental weight (in all simple Lie algebras except of type $A_n$). 
However our proof is a simple consequence of our methods in section \ref{section2}.
Theorem \ref{thm-1_2} has the following Corollary.
\begin{corollary} \label{cor-1_3} Let $L= \Lie(G)$ be the Lie algebra of a compact semisimple Lie group $G$.
Let $C$ be a Cartan subalgebra of $L$ and let $C^{\perp}$ be the orthogonal complement of $C$
(with respect to negative of the Killing form on $L$). Then the $G$-orbit of $C^{\perp}$ is all of $L$. \end{corollary}
We note that Corollary \ref{cor-1_3} can be easily obtained directly from Kostant's convexity
Theorem (see Appendix \ref{sectionA}). In fact, Corollary \ref{cor-1_3} was essentially the key step in 
Karl-Hermann Neeb's proof of Goto's Theorem, additive version.
In section \ref{section4}, we present two “applications” of Corollary \ref{cor-1_3}. First, Corollary \ref{cor-1_3} answers a
question by L. Florit and W. Ziller on fatness of certain principal bundles as follows
\begin{corollary} \label{cor-1_4} If $G$ is a compact semisimple Lie group, endowed with minus its Killing form,
and $G \hookrightarrow P \to B$ is a principal $G$-bundle on $B$ with total space $P$, endowed with a
connection $1$-form $\theta : TP \to \mathfrak{g}$, then $\theta$ is fat if and only if $C^{\perp}$ is fat, where
$C$ can be chosen to be any Cartan subalgebra of $\mathfrak{g}$. \end{corollary}
Our second application of Corollary \ref{cor-1_3} is the following.
\begin{corollary} \label{cor-1_5} Let $L= \Lie(G)$ be the Lie algebra of a compact semisimple Lie group $G$.
Then the commutator map $L \times L \to L$ is open at $(0, 0)$. \end{corollary}
This fact was obtained in \cite{A-M} using their Lemma 2.2 which was proved by using some tables
about fundamental weights as explained above. (Recall that their Lemma 2.2 is our
Theorem \ref{thm-1_2} above). However, after a careful study of their proof, we found out that
one only needs the partial result stated in Corollary \ref{cor-1_3} above.
In section \ref{section5_1}, we make the following remark.

\begin{corollary} \label{corollary-1_6} Let $L$ be a compact semisimple Lie algebra, and let $a$ be a regular 
element of $L$. Then there exists a regular element $b$ in $L$, such that 
\[ L = [a, L] + [b, L] \]
and $a$, $b$ are orthogonal (in fact the centralizers of $a$ and $b$ in $L$ are also orthogonal). \end{corollary}
More generally, we have
\begin{corollary} \label{corollary-1_7} Let $L = \mathfrak{k} \oplus \mathfrak{p}$ be a Cartan decomposition 
of a semisimple Lie algebra $L$ (where $\mathfrak{k}$ is a maximal compact Lie subalgebra) such that $L$ is non-Hermitian of 
full rank. That is, $\mathfrak{k}$ is semisimple and $\rank(\mathfrak{k}) = \rank(L)$. Let $a$ be a regular element 
of $\mathfrak{k}$. Then there exists a regular element $b$ in $\mathfrak{k}$, such that $L = [a, L] + [b, L]$ and 
$a$, $b$ are orthogonal. \end{corollary}
In section \ref{section5_1}, we make the following trivial remark.
\begin{remark} \label{rem-1_8} The following two conjectures are equivalent. \begin{enumerate}
\item[1)] Every element $x$ in a real semisimple Lie algebra $L$ is the commutator of two elements where one
element can be chosen to be regular (hence semisimple).
\item[2)] Every element $x$ in a real semisimple Lie algebra $L$ is orthogonal to some Cartan subalgebra of $L$.
\end{enumerate} 
\end{remark}
In section \ref{section5_2}, we survey most of the results of D. Akhiezer in his recent interesting paper
\cite{Akh} that shows that the above conjecture is valid in many real simple Lie algebras (see Theorem \ref{thm-5_2_1})
In section \ref{section5_3}, we survey one result by G. Bergman and N. Nahlus in \cite{B-N} and another result in
preparation by the second author in \cite{N} related to $1.5$ generators of simple Lie algebras. For
example, in any simple Lie algebra $L$ over a field of characteristic $0$, and for any $x$ in $L\setminus\{0\}$,
there exists a regular (hence semisimple) element $y$ in $L$ such that $L=[L, x] + [L, y]$ (\cite{N}). This 
last property is related to the concept of $1.5$ generators as follows: if $x$ and $y$ generate $L$ as a 
Lie algebra then $L = [L,x] + [L,y]$ by \cite[Lemma 25c]{Bor1}.

Finally in section \ref{section6}, we construct explicit examples of orthogonal Cartan subalgebras in the cases
of $\mathfrak{su}(n)$, $\mathfrak{sp}(n)$, based on the idea of circulant matrices (cf. \cite{K-S}), and in 
the case of $\mathfrak{so}(n)$.

\acknowledge{The authors are grateful to Anthony Knapp for his help in answering some questions about the 
general theory of non-compact semisimple Lie algebras. We are also indebted to Karl-Hermann Neeb for his encouragement 
and interesting discussions about two results in the paper. Finally, we are also indebted to Wolfgang Ziller for 
kindly pointing out that our Corollary \ref{cor-1_3} answers a question about fatness in principal bundles.}

\section{Goto's theorem by Coxeter transformations} \label{section2}

In this section, we give a proof of Goto's theorem, additive version, by using Coxeter transformations.

Let $L = \Lie(G)$ be the Lie algebra of a compact semisimple Lie group. Let $C$ be a maximal toral subalgebra of $L$. 
Let $\Sigma \subseteq C^*$ (the dual of $C$) be the root system of $G$ with respect to $C$ and we fix an 
ordering of the roots, with $\Delta$ as the set of positive roots, and let $\{a_1,\ldots,a_n\}$ be the set of 
simple positive roots.

\begin{lemma} \label{lemma-2_1} Let $\{s_1,\ldots,s_n\}$ be the Weyl reflections corresponding to the simple positive roots 
$\{a_1,\ldots,a_n\}$. Then, for any permutation of $\{1,\ldots,n\}$, the Coxeter transformation 
$c=s_1 s_2 \cdots s_n$ has no fixed points, other than $0$. That is, $1$ is not an eigenvalue of $c$. 
\end{lemma}

See the very short proof in \cite{Hum2}, p.$76$.

\begin{lemma} \label{lemma-2_2} (cf. \cite{H-M1}, Lemma $6.53$) Let $L = \Lie(G)$ be the Lie algebra of a compact 
semisimple Lie group. Let $C = \Lie(T)$ be a maximal toral subalgebra. Let $c$ be any Coxeter transformation 
with respect to $(L,C)$ as in lemma \ref{lemma-2_1}. $c = \Ad(n)$ for some $n \in N(T)$, the normalizer of 
$T$ in $G$, since $c$ is in the Weyl group of $(G,T)$. Then
\[ \Ad(n)|_C - \Id: C \to C \]
is an isomorphism.
\end{lemma}
\begin{proof} Immediate from lemma \ref{lemma-2_1}. \end{proof}
\begin{theorem} \label{thm-2_3} In the setting of lemma \ref{lemma-2_2}, suppose $n=\exp(N)$ for some $N$ in $L = \Lie(G)$ (which 
exists since the exponential map is surjective in connected compact semisimple Lie groups). Then $C$ is contained 
in $[L,N]$. \end{theorem}
\begin{proof} Let $x \in C = \Lie(T)$. Then by lemma \ref{lemma-2_2}, $x = \Ad(n)(t) - t$ for some $t$ in $C$. But the exponential map 
is surjective in compact connected semisimple Lie groups, so $n = \exp(N)$, for some $N \in L$. Hence 
\[ x = \Ad(n)(t) - t = \Ad(\exp(N))(t) - t = \exp(\ad(N))(t) - t \]
But $x = \exp(\ad(N))(t) - t$ can be written as $[N, y(x)]$, for some $y(x) \in L$ (by factorization) through 
writing the exponential series as the limit of its partial sums. Hence $x = [N, y(x)]$.
\end{proof}
Next we strengthen Theorem \ref{thm-2_3} to replace $N$ by a regular element of $L$.

\begin{theorem} \label{thm-2_4} Let $L = \Lie(G)$ be the Lie algebra of a compact semisimple Lie group $G$, 
and let $C$ be a Cartan subalgebra of $L$. Then there exists a regular element $a$ of $L$ such that 
$C \subseteq [a, L]$. In particular, every element $x$ of $L$ can be written as $x = [a, b]$ where $a$ can be chosen to be regular.
\end{theorem}
\noindent \textit{Comment:} for comparison, the proof by Kostant's convexity Theorem in the Appendix shows that $C^{\perp} \subseteq [x, L]$ 
for some regular element $x$ of $L$.
  
\begin{proof} By Theorem \ref{thm-2_3}, $C \subseteq [N, L]$. Since $N$ is semisimple, $N$ is contained in a Cartan subalgebra 
$C’$ of $L$, so $C’ \subseteq Z(N)$. 
Now $C’=Z(a)$ for some regular element $a$ of $L$.
So $Z(a)  \subseteq Z(N)$.  
Hence $Z(N)^{\perp} \subseteq Z(a)^{\perp}$.
But $Z(N)^{\perp} = [N, L]$ and  $Z(a)^{\perp} = [a, L]$ by Remark \ref{remark-5_1_2}.
Hence $[N, L] \subseteq [a, L]$. So $C \subseteq [a, L]$ as desired.
The last conclusion follows from the fact that every element of $L$ belongs to a Cartan subalgebra, 
and all Cartan subalgebras of $L$ are conjugate. \end{proof}

\begin{example} \label{ex-2_5} Consider $\mathfrak{su}(2k+1)$, which is a compact semisimple Lie algebra. It 
can be shown by a simple calculation that 
\[ \mathbf{n} = \left( \begin{array}{ccccc} 0 & 0 & \cdots & 0 & 1 \\
1 & 0 & \cdots & 0 & 0 \\
\vdots & \vdots & \ddots & \vdots & \vdots \\
0 & 0 & \cdots & 0 & 0 \\
0 & 0 & \cdots & 1 & 0 \end{array} \right) \]
is an element of $N(T)$ which gives modulo $T$ a Coxeter transformation of $SU(2k+1)$. Consider the unitary change of basis matrix 
$\mathbf{g} = (\mathbf{g}_{ab})$ for $1 \leq a,b \leq 2k+1$, with
\[ \mathbf{g}_{ab} = \frac{1}{\sqrt{2k+1}}\gamma_k^{(a-1)(b-1)} \]
where $\gamma_k = \exp(\frac{2 \pi i}{2k+1})$. Let $\mathbf{D} = \diag(1, \bar{\gamma}, \bar{\gamma}^2, \cdots, \bar{\gamma}^{2k})$. 
Then a simple calculation shows that $\mathbf{n} = \mathbf{g} \mathbf{D} \mathbf{g}^{-1}$. So we let 
\[ \mathbf{\Lambda} = \diag(0 , -ic_k, -2ic_k, \cdots, -kic_k, kic_k, (k-1)ic_k,\cdots, ic_k ), \] 
where $c_k = \frac{2\pi}{2k+1}$. It is clear that $\Lambda$ is an element of $\mathfrak{su}(2k+1)$ and that 
$\exp(\mathbf{\Lambda}) = \mathbf{D}$. Hence we obtain that 
\[ \mathbf{N} = \mathbf{g} \mathbf{\Lambda} \mathbf{g}^{-1} \]
is an element of $\mathfrak{su}(2k+1)$ which satisfies $\exp(\mathbf{N}) = \mathbf{n}$.  
\end{example}
\begin{example} \label{ex-2_6} The case of $\mathfrak{su}(2k)$ can be dealt with in a similar fashion. We obtain that
\[ \mathbf{n} = \left( \begin{array}{ccccc} 0 & 0 & \cdots & 0 & -1 \\
1 & 0 & \cdots & 0 & 0 \\
\vdots & \vdots & \ddots & \vdots & \vdots \\
0 & 0 & \cdots & 0 & 0 \\
0 & 0 & \cdots & 1 & 0 \end{array} \right) \]
is an element of $N(T)$ which gives modulo $T$ a Coxeter transformation of $SU(2k)$. Consider the unitary change of basis given by
$\mathbf{g} = (\mathbf{g}_{ab})$, where
\[ \mathbf{g}_{ab} = \frac{1}{\sqrt{2k}} \beta_k^{(a-1)(2b-1)} \]
where $\beta_k = \exp(\frac{\pi i}{2k})$. Also, let $\mathbf{D} = (\mathbf{D}_{ab})$, where
\[ \mathbf{D}_{ab} = \exp \left(- \frac{(2a-1)\pi i}{2k} \right) \delta_{ab} \]
Then a simple calculation shows that $\mathbf{n} = \mathbf{g} \mathbf{D} \mathbf{g}^{-1}$. Finally, let 
\[ \mathbf{\Lambda} = \diag(- id_k, - 3id_k, \cdots, - (2k-1)id_k, (2k-1)id_k, 
(2k-3)id_k, \cdots, id_k) \]
where $d_k = \frac{\pi}{2k}$.
As in example \ref{ex-2_5}, $\mathbf{N} = \mathbf{g} \mathbf{\Lambda} \mathbf{g}^{-1}$ is an element in $\mathfrak{su}(2k)$ 
satisfying $\mathbf{n} = \exp(\mathbf{N})$.
\end{example}

\section{Orthogonal Cartan subalgebras} \label{section3}

In this section, we prove the following non-trivial theorem which was proved in \cite{A-M} lemma 2.2, via the 
classification tables of all fundamental weights. Their proof relies on the observation found in the classification 
tables \cite{Bor2}, that the highest root is equal or twice some fundamental weight, in all simple 
Lie algebras except those of type $A_n$. It would be nice to have a more direct argument for this nice result, and this is what 
we shall present in this section.

\begin{theorem} \label{thm-3_1} Let $L = \Lie(G)$ be the Lie algebra of a compact semisimple Lie group. Then $L$ has orthogonal 
Cartan subalgebras (where one of them can be chosen arbitrarily)
\end{theorem}
\begin{proof} Let $C$ be a Cartan subalgebra of $L$. Then by Theorem \ref{thm-2_4},
$C \subseteq [a, L]$ for some regular element $a$ of $L$. Since $a$ regular in $L$, $Z(a)$ is a Cartan subalgebra 
of $L$. Moreover, $Z(a)^{\perp} = [a, L]$ by Remark \ref{remark-5_1_2}.
Hence $C \subseteq Z(a)^{\perp}$. Thus the Cartan subalgebras $C$ and $Z(a)$ are orthogonal.
Finally, $C$ can be chosen arbitrarily since all Cartan subalgebras (in Lie algebras of compact semisimple Lie groups) 
are conjugate and the Killing form is invariant under automorphisms of $L$ \cite[Prop. 1.119]{Knp}. This proves 
the theorem.
\end{proof}

\begin{theorem} \label{thm-3_2} Let $L = \Lie(G)$ be the Lie algebra of a compact semisimple Lie group. 
Let $C$ be a given Cartan subalgebra, and  let $C^{\perp}$ be the orthogonal of $C$ 
(with respect to the negative of the Killing form on $L$). Then for every element $x$ of $L$,  
there exists an element $g$ in $G$ , such that $g.x$ belongs to $C^{\perp}$. That is, $C^{\perp}$ intersects every $G$-orbit of $L$.
\end{theorem}
\begin{proof} By Theorem \ref{thm-3_1}, there exists a Cartan subalgebra $C’$ which is orthogonal to $C$.
For every element $x$ in $L$, we know that there exist a $g$ in $G$ such that $g.x \in C’$.
Hence $g.x$ is orthogonal to $C$, so $g.x$ belongs to $C^{\perp}$. Thus, $C^{\perp}$ intersects every $G$-orbit in $L$. 
(Or equivalently, $L$ is the union of all $g.C^{\perp}$ as $g$ varies over $G$). \end{proof} 

As noted in the introduction, Theorem \ref{thm-3_2} can be obtained almost immediately from Kostant's 
Convexity Theorem (See Appendix, Theorem \ref{thm-A_3}).

\begin{corollary} \label{corollary-3_3} Let $L$ be a compact semisimple Lie algebra and let $a$ be a given 
regular element of $L$. Then there exists a regular element $b$ in $L$, such that
$L = [a, L] + [b, L]$ and $b$ is orthogonal to $a$ (under the negative of the Killing form of $L$).
\end{corollary}

\begin{proof} By Theorem \ref{thm-3_1},  there exist orthogonal Cartan subalgebras $C$ and  $D$ in $L$, so
$D \subseteq C^{\perp}$. Since $L = D \oplus D^{\perp}$,  it follows that 
$L = C^{\perp} + D^{\perp}$. Let $x$ be any regular element of $L$ in $C$ and let $y$ be any regular element of 
$L$ in $D$, so $C = Z(x)$ and $D = Z(y)$. Thus $L = Z(x)^{\perp} + Z(y)^{\perp}$.
But $Z(x)^{\perp} =[x, L]$ and $Z(y)^{\perp} = [y, L]$ by Remark \ref{remark-5_1_2}. 
Hence $L = [x, L] + [y, L]$. Note that $x$ and $y$ are regular in $L$ and  orthogonal (since they belong to 
$C$ and $D$ which are orthogonal).

The given regular (semisimple) element $a$ belongs to a Cartan subalgebra and all Cartan subalgebras in $L$ 
are conjugate. Hence there exists an inner automorphism $f$ such that $f(a)$ is in $C$.
Since $a$ and hence $f(a)$ is regular in $L$, the element $x$ in the previous paragraph can be chosen to be 
$f(a)$. Hence $L = [f(a), L] + [y, L]$. Now we apply $f^{-1}$ on both sides to get
$L = [a, L] + [f^{-1}(y), L]$. Since $x = f(a)$ and $y$ are orthogonal, then so are
$a$ and $f^{-1}(y)$ since automorphisms preserve orthogonality \cite[Prop. 1.119]{Knp}.
Finally, since $y$ is regular, then so is $f^{-1}(y)$. \end{proof}

Corollary \ref{corollary-3_3} can be easily generalized to the case of non-Hermitian real semisimple 
Lie algebras of full rank.

\begin{corollary} \label{corollary-3_4} Let $L = \mathfrak{k} \oplus \mathfrak{p}$ be a Cartan decomposition of a 
semisimple Lie algebra $L$ (where $\mathfrak{k}$ is the compact Lie algebra) such that $L$ is non-Hermitian of full rank. 
That is, $\mathfrak{k}$ is semisimple and $\rank(\mathfrak{k}) = \rank(L)$. Let $a$ be a regular 
element of $\mathfrak{k}$. Then there exists a regular element $b$ in $\mathfrak{k}$, such that $L = [a, L] + [b, L]$ 
and $b$ is orthogonal to $a$. Moreover the centralizers $Z(a)$ and $Z(b)$ of $a$ and $b$ in $L$ are 
also orthogonal.\end{corollary}

\begin{proof} The proof is very similar to the proof of Corollary \ref{corollary-3_3} with the evident modifications  
as follows. By Theorem \ref{thm-3_1}, there exist orthogonal Cartan subalgebras $C$ and $D$ in $\mathfrak{k}$.
In fact, $C$ and $D$ are also orthogonal with respect to the Killing form on $L$ (see Remark \ref{remark-3_5} below). 
Since $\rank(\mathfrak{k}) = \rank(L)$, $C$ and $D$ are also Cartan subalgebras of $L$.
Since the restriction of the Killing form of $L$ on every Cartan subalgebra is non-degenerate, we have
$L = D + D^{\perp}$. But $D$ is a subset of $C^{\perp}$ in $L$.
Hence $L = C^{\perp} + D^{\perp}$. Let $x$ be any regular element of $L$ in $C$ and let $y$ be any regular element of $L$ in $D$.
Then the rest of the proof is verbatim as in the proof of Corollary \ref{corollary-3_3}, leading to  
$L = [x, L] + [y, L]$ and finally, $L = [a, L] + [f^{-1}(y), L]$ where $a$ and $f^{-1}(y)$ are regular and 
orthogonal in $L$. \end{proof}

Apart from the Lie algebras of compact semisimple Lie groups, some other examples of Lie algebras 
satisfying the hypotheses of Corollary \ref{corollary-3_4} are $\mathfrak{so}(2j,2k)$ with $j$ and $k$ both $>1$, 
$\mathfrak{so}(2j, 2k+1)$ if $j>1$ and $\mathfrak{sp}(p,q)$ where $p$ and $q$ are both positive.

\vspace*{1\baselineskip}

\begin{remark} \label{remark-3_5} Over a field of characteristic $0$, let $S$ be a semisimple
Lie subalgebra of a semisimple Lie algebra $L$, and let $C$ be a Cartan
subalgebra of $S$. Then the orthogonal space of $C$ in $S$ (with respect to the 
Killing form of $S$) is contained in the orthogonal space of $C$ in $L$
(with respect to the Killing form of $L$).
\end{remark}

\begin{proof} We know that $C = Z(a)$ for some regular element $a$ of $S$.
Moreover, the restriction of the Killing form of $L$ to $S$ is also non-degenerate
by a simple application of Cartan's solvability criterion by \cite[chapter 1, section 6.1, Proposition 1]{Bor1}.   
Hence, by \cite[Lemma 2.2]{Akh} (stated after Remark \ref{remark-5_1_2} below), the orthogonal space of $C = Z(a)$ in 
$S$ coincides with $[a, S]$ independently of the invariant non-degenerate symmetric bilinear form on $S$.
\end{proof}

\section{Applications} \label{section4}
\subsection{Openness of commutator maps} \label{section4_1}

We start with the following immediate application of Theorem \ref{thm-3_2}.

\begin{lemma} \label{lemma-4_1_1} Let $L = \Lie(G)$ be the Lie algebra of a compact semisimple Lie group, and let 
$C = \Lie(T)$ be a maximal toral subalgebra of $L$. Let $V$ be a $G$-stable neighborhood of $0 \in L$. Then
\[ G.(V \cap C^{\perp}) = V \]
\end{lemma}
\begin{proof} This follows immediately from Theorem \ref{thm-3_2} as follows. Since $V$ is $G$-stable, 
$G.(V \cap C^{\perp}) \subseteq V$. Conversely, if $v \in V$, then by Theorem \ref{thm-3_2}, there exists a 
$g \in G$ such that $g.v \in C^{\perp}$. Hence $g.v \in V \cap C^{\perp}$; so $v \in G.(V \cap C^{\perp})$.
\end{proof}
Using lemma \ref{lemma-4_1_1}, one can prove the following theorem.
\begin{theorem} \label{thm-4_1_2} (\cite{A-M}, theorem $2.1$). Let $L = \Lie(G)$ be the Lie algebra of a compact 
semisimple Lie group. Then the commutator map $\Comm: L \times L \to L$ given by $\Comm(x,y) = [x,y]$ is open 
at $(0,0)$.
\end{theorem}
The proof of this theorem as given in \cite{A-M} relies on their lemma $2.2$, that we proved in Theorem \ref{thm-3_1} 
above (without the classification tables of the fundamental weights versus the highest roots). 

However, a careful study of their proof of Theorem \ref{thm-4_1_2} only requires the partial result 
in our lemma \ref{lemma-4_1_1}. 

\begin{remark} \label{rmk-4_1_3} It is worth noting that the authors in \cite{A-M} were able to use the fact that 
the commutator map $\Comm: L \times L \to L$ given by $\Comm(x,y) = [x,y]$ is open at $(0,0)$ to prove their main 
result that $\Comm_G: G \times G \to G$ given by $\Comm_G(g_1,g_2) = [g_1,g_2] = g_1g_2g_1^{-1}g_2^{-1}$ is open 
at $(\Id,\Id)$ in $G \times G$ (\cite{A-M}, Corollary 3.1).
\end{remark}

\subsection{Fatness of certain connections on principal fiber bundles} \label{section4_2}

We first review the basic definition of fatness, in the context of principal fiber bundles with connection. Such 
material can easily be found in the literature (see the original paper \cite{Wei} by A. Weinstein, or later works such 
as \cite{F-Z} or \cite{Ler}). Let $G \hookrightarrow P \stackrel{\pi}{\rightarrow} B$ be a principal $G$-bundle, with base 
manifold $B$, and total space $P$. Let $\theta: TP \to \mathfrak{g}$ be a connection $1$-form on $P$, where $TP$ denotes 
the tangent bundle of $P$. We shall denote the curvature of $\theta$ by $\Omega(-,-): \mathcal{H} \times \mathcal{H} \to \mathfrak{g}$, 
where $TP = \mathcal{H} \oplus \mathcal{V}$ is the decomposition into horizontal and vertical subbundles induced by $\theta$. 
Fix a bi-invariant (positive-definite) inner product $B(-,-): \mathfrak{g} \times \mathfrak{g} \to \mathbb{R}$ on $\mathfrak{g}$.
 
\begin{definition} \label{def-4_2_1} In the setting above of a principal $G$-bundle $P$ with connection $\theta$, 
an element $u \in \mathfrak{g}$ is said to be fat if 
$\Omega_u: \mathcal{H} \times \mathcal{H} \to \mathbb{R}$ is non-degenerate, where $\Omega_u$ is defined by
\[ \Omega_u(Y_1, Y_2) = B(\Omega(Y_1,Y_2), u) \]  
We also say that a subspace $V \subseteq \mathfrak{g}$ is fat if every non-zero element of $V$ is fat. The 
connection $\theta$ is said to be fat if $\mathfrak{g}$ is fat. 
\end{definition}

We can easily deduce the following theorem, answering a question raised in \cite[Remark 3.13]{F-Z}
\begin{theorem} \label{thm-4_2_2} If $G$ is a compact semisimple Lie group, $B(-,-)$ is minus its Killing form, and 
$G \hookrightarrow P \stackrel{\pi}{\to} B$ is a principal $G$-bundle on $B$ with total space $P$, endowed with a connection $1$-form 
$\theta: TP \to \mathfrak{g}$, then $\theta$ is fat if and only if $C^{\perp}$ is fat, where $C$ can be chosen to be 
any Cartan subalgebra of $\mathfrak{g}$.
\end{theorem}
\begin{proof} The ``only if'' part is trivial. The other direction follows from the fact that $G.C^{\perp} = \mathfrak{g}$ by 
\ref{cor-1_3} (which follows immediately from Kostant's convexity theorem), and the known fact that if $u \in \mathfrak{g}$ is 
fat, then $g.u$ is also fat, for any $g \in G$ (see for example \cite{F-Z}).
\end{proof} 

\section{\texorpdfstring{Commutators in semisimple Lie algebras 
and survey of some related results}{Conjectures and survey}} \label{section5}

\subsection{\texorpdfstring{Conjectures about commutators in semisimple Lie algebras}
{Some conjectures about commutators}} \label{section5_1}

\begin{remark} \label{remark-5_1_1} The following two conjectures are equivalent.
\begin{enumerate}[leftmargin=*]
\item[1)] Every element $x$ in a real semisimple Lie algebra $L$ is the commutator of two elements where one 
element can be chosen to be regular (hence semisimple).
\item[2)] Every element $x$ in a real semisimple Lie algebra $L$ is orthogonal to some Cartan subalgebra of $L$.
\end{enumerate}
\end{remark}

\begin{proof} Note that every regular element of a semisimple Lie algebra (over a field of characteristic $0$) is 
semisimple \cite[Chapter 7, section 2, Corollary 2]{Bor2}. First we prove conjecture 1) implies conjecture 2). 
Suppose  $x = [a, b]$ where $a$ is regular (semisimple). Then $x$ belongs to $[a, L]$.
But $[a, L]= Z(a)^{\perp}$ by Remark \ref{remark-5_1_2}. Hence $x$ belongs to $Z(a)^{\perp}$.
Now $a$ is regular (semisimple), $Z(a)$ is a Cartan subalgebra $C$ of $L$.
Consequently, $x$ belongs to $C^{\perp}$.
Conversely, if $x$ belongs to $C^{\perp}$ where $C$ is a Cartan subalgebra of $L$.
Then $C =Z(a)$ for some regular element $a$ of $L$.
So $x$ belongs to $Z(a)^{\perp}$ which coincides with $[a, L]$ by Remark \ref{remark-5_1_2}.
Hence $x$ belongs to $[a, L]$. 
\end{proof}

\begin{remark} \label{remark-5_1_2} Let $a$ be a semisimple element of a semisimple Lie algebra $L$ over a field of 
characteristic $0$.  Then $Z(a)^{\perp} =  [a, L]$ with respect to the negative of the Killing form on $L$ (where $Z(a)$ is 
the centralizer of $a$ in $L$.
\end{remark}
Remark \ref{remark-5_1_2} is part of Theorem 4.1.6 in \cite{Var} which can be checked directly. Roughly, $Z(a)$ is 
orthogonal to $[a, L]$, and both $Z(a)^{\perp}$ and $[a, L]$ have the same dimension. However, it is worth noting that 
Remark \ref{remark-5_1_2} (with essentiall the same proof) is also valid for any element of $L$ and, for any invariant 
non-degenerate symmetric bilinear form on $L$ as shown in \cite[Lemma 2.1]{Akh}.
 
\subsection{Survey of some recent results on the commutator conjecture} \label{section5_2}

In the recent paper \cite{Akh}, the commutator conjecture was proved in many cases.

\begin{theorem} \cite[Thms 1.1--1.4]{Akh} \label{thm-5_2_1} Let $L = \Lie(G)$ be a simple real Lie algebra. Let 
$L = \mathfrak{k} \oplus \mathfrak{p}$ be a Cartan decomposition of $L$ where $\mathfrak{k}$ is the 
reductive part. Let $\mathfrak{a}$ be a Cartan subspace of $\mathfrak{p}$, and let $\mathfrak{m}$ 
be the centralizer of $\mathfrak{a}$ in $\mathfrak{k}$. Then every element $x$ of $L$ can be written as 
$x = [a,b]$ where $a$ can be chosen to be a regular (semisimple) element of $L$, in each of the following 
cases:
\begin{enumerate}
\item $L$ is compact or split (i.e. $\mathfrak{m} = 0$)
\item $L$ is non-hermitian (i.e. $\mathfrak{k}$ is semisimple) and $\rank(\mathfrak{k}) = \rank(L)$
\item $\mathfrak{m}$ is a semisimple Lie algebra
\end{enumerate}
\end{theorem}
The paper \cite{Akh} relies on the Kostant convexity theorem and some interesting arguments by the author.

For more information about complex or split semisimple Lie algebras, see \cite[Thm. 4.3]{A-M}  and \cite[Lemma 2]{Brn}.

\subsection{Survey on generators and $1.5$ generators of simple Lie algebras} \label{section5_3}

\begin{theorem} \label{thm-5_3_1} \cite[Thm. 26]{B-N} Let $L$ be a semisimple Lie algebra over a field of characteristic $0$ (or characteristic 
$p > 3$). Then there exist $x$, $y$ in $L$ such that $L = [x,L] + [y,L]$.
\end{theorem}

\begin{theorem} \label{thm-5_3_2} \cite[Thm. 2]{N} (on $1.5$ generators). Let $L$ be a semisimple Lie algebra over 
any field of characteristic $0$. For any fixed $x \in L \setminus \{0\}$, there exists a regular element $y \in L$ 
such that
\[ L = [x,L] + [y,L] \]
\end{theorem}

We note that the above two theorems are related respectively to the concepts of 
generators and $1.5$ generators in Lie algebras in view of the following fact:
If $x$ and $y$ generate, in the Lie algebra sense, a Lie algebra $L$,  then 
$L = [L, x] + [L, y]$ by  \cite[Lemma 25c]{Bor1}.

\section{Examples of orthogonal Cartan subalgebras} \label{section6}

In the examples of orthogonal Cartan subalgebras we will see, corresponding to the classical groups, circulant matrices play a 
special role. We begin by reviewing circulant matrices, following the presentation in \cite{K-S}. Given 
a vector $\mathbf{a} = (a_0,a_1,\ldots,a_n) \in \mathbb{C}^n$, we define
\[ A = \left( \begin{array}{ccccc} a_0 & a_1 & \cdots & a_{n-2} & a_{n-1} \\
       a_{n-1} & a_0 & \cdots & a_{n-3} & a_{n-2} \\
       \vdots & \vdots & \ddots & \vdots & \vdots \\
       a_2 & a_3 & \cdots & a_0 & a_1 \\
       a_1 & a_2 & \cdots & a_{n-1} & a_0 \end{array} \right)
\]
Let $\epsilon = e^{\frac{2\pi i}{n}}$, which is a primitive $n$'th root of unity, and let
\[ x_l = \frac{1}{\sqrt{n}}(1,\epsilon^l, \epsilon^{2l}, \ldots, \epsilon^{(n-1)l}) \in \mathbb{C}^n \]
for $l = 0,\ldots,n$. We let
\[ U = (x_0,x_1, \ldots, x_{n-1}) \]
In other words, $U$ is the square matrix having as columns $x_0$, $x_1$,$\ldots$, $x_{n-1}$. It is easy to check that 
$U$ is both unitary and symmetric. We introduce the numbers
\[ \lambda_l = \sum_{j=0}^{n-1} \epsilon^{lj} a_j \]
for $l = 1,2,\ldots, n$. Let $\Lambda = \diag(\lambda_0,\lambda_1,\ldots,\lambda_{n-1})$. A calculation shows that
\[ A = U \Lambda U^{-1} \]
so that the $\lambda_l$s are the eigenvalues of $A$. An important observation is that $U$ does not depend on $A$, 
so that the same unitary matrix $U$ diagonalizes \textbf{all} 
complex circulant $n$ by $n$ matrices. In particular, this also implies that the space $\Circ(n)$ of all complex circulant $n$ 
by $n$ matrices is \textbf{abelian} (which can be seen directly too).

\subsection{The Lie algebras $\mathfrak{su}(n)$}

\begin{theorem} \label{thm-6_1} If $C$ is the Cartan subalgebra of $\mathfrak{su}(n)$ consisting zero trace skew-hermitian 
diagonal $n$ by $n$ complex matrices, and
\[ C' = \Circ(n) \cap \mathfrak{su}(n) \]
then $C$ and $C'$ are two orthogonal Cartan subalgebras of $\mathfrak{su}(n)$, with respect to minus the Killing form of 
$\mathfrak{su}(n)$.
\end{theorem}
\begin{proof} We already know that $C'$ is abelian, since $\Circ(n)$ is abelian, so it only remains to prove that 
$\dim_{\mathbb{R}}(C') = n-1$. This is so because $C'$ corresponds to vectors $\mathbf{a}$ having
\[ \left\{ \begin{array}{rcl} a_0 & = & 0 \\
               a_{n-l} & = & -\bar{a_l}, \qquad \text{ for $1 \leq l \leq n-1$ }
           \end{array} \right. \]
from which it follows indeed that the real dimension of $C'$ is $n-1$.
\end{proof}

\subsection{The Lie algebras $\mathfrak{sp}(n)$}

It is well known that $\mathfrak{gl}(n,\mathbb{H}) \subseteq \mathfrak{gl}(2n,\mathbb{C})$. This relies on the observation that a 
quaternion $x$ can be written in a unique way as $x = u+jv$, where $u$ and $v$ are complex numbers. It can also be verified that the map 
\[ x \mapsto \left( \begin{array}{cc} u & -\bar{v} \\
                                  v & \bar{u} \end{array} \right) \]
is an injective algebra homomorphism from $\mathbb{H}$ into $\mathfrak{gl}(2,\mathbb{C})$.

The Lie algebras $\mathfrak{sp}(n)$ can be written as
\[ \mathfrak{sp}(n) = \mathfrak{gl}(n,\mathbb{H}) \cap \mathfrak{u}(2n) \]
where both Lie algebras on the right-hand side are understood as Lie subalgebras of $\mathfrak{gl}(2n,\mathbb{C})$, using 
the remark above. We then have the following theorem.

\begin{theorem} \label{thm-6_2} If $C$ is the Cartan subalgebra of $\mathfrak{sp}(n)$ consisting of complex diagonal pure imaginary 
$2n$ by $2n$ matrices, and
\[ C' = \Circ(2n) \cap \mathfrak{sp}(n) \]
where both Lie algebras on the right-hand side are thought of as Lie subalgebras of $\mathfrak{gl}(2n,\mathbb{C})$ in the 
usual way, then $C$ and $C'$ are two orthogonal Cartan subalgebras of $\mathfrak{sp}(n)$, with respect to minus the Killing form of 
$\mathfrak{sp}(n)$.
\end{theorem} 
\begin{proof}The only thing to check is that $\dim_{\mathbb{R}}(C') = n$. This essentially follows from the fact that $\Circ(2n)$ is a complex 
$2n$-dimensional subspace of $\mathfrak{gl}(2n,\mathbb{C})$, which admits the decomposition
\[ \Circ(2n) = C' \oplus iC' \oplus jC' \oplus kC' \]
the latter being a \emph{real} vector space decomposition, from which it can be deduced that $\dim_{\mathbb{R}}(C') = n$, as claimed.  
\end{proof}

\subsection{The Lie algebras $\mathfrak{so}(n)$}

Define
\begin{align*} C_{2k} &= \left\{ \left( \begin{array}{cc} 0_k & \Lambda \\
                                             -\Lambda & 0_k \end{array} \right): \text{ $\Lambda$ is a real $k$ by $k$ diagonal 
matrix} \right\} \\
C_{2k+1} &= \left\{ \left( \begin{array}{ccc} 0_k & \Lambda & 0_{k,1} \\
                                             -\Lambda & 0_k & 0_{k,1} \\
                                              0_{1,k} & 0_{1,k} & 0 \end{array} \right): \text{ $\Lambda$ is a real $k$ by $k$ diagonal 
matrix} \right\}
\end{align*}
Here $0_k$ denotes the zero $k$ by $k$ matrix, while for instance $0_{j,k}$ denotes the $j$ by $k$ zero matrix. It is easy to 
check that $C_n$ is a Cartan subalgebra of $\mathfrak{so}(n)$ for all $n \geq 3$. Our plan in this section is 
to construct for each $n$, a Cartan subalgebra $C'_n$ which is orthogonal to $C_n$. Our next lemma will allow us to restrict our 
attention to $\mathfrak{so}(n)$, with $n$ even.

\begin{lemma} \label{lemma-6_3} if $C'_{2k}$ is a Cartan subalgebra of $\mathfrak{so}(2k)$ orthogonal to $C_{2k}$, and if 
$\iota: \mathfrak{so}(2k) \hookrightarrow \mathfrak{so}(2k+1)$ denotes the inclusion of Lie algebras which satisfies $\iota(C_{2k}) = 
C_{2k+1}$, then $\iota(C'_{2k})$ is a Cartan subalgebra of $\mathfrak{so}(2k+1)$ which is orthogonal to $C_{2k+1}$.
\end{lemma}
\begin{proof} This follows from the fact that $\mathfrak{so}(2k)$ and $\mathfrak{so}(2k+1)$ both have rank $k$, and 
that the restriction of the Killing form of $\mathfrak{so}(2k+1)$ to $\mathfrak{so}(2k)$ is a constant multiple of the 
Killing form of $\mathfrak{so}(2k)$. \end{proof} 

The next lemma will enable us to restrict our attention further to just $\mathfrak{so}(4)$ and $\mathfrak{so}(6)$. But first, we 
introduce the following notation. If $X$ is a $k$ by $k$ matrix, we denote by $\tilde{X}$ the following $k+2$ by $k+2$ matrix
\[ \tilde{X} = X \oplus 0_2 \]
Also, if $x$ is $2k$ by $2k$ matrix, one can write $x$ as
\[ x = \left( \begin{array}{cc} A & B \\
                               C & D \end{array} \right) \]
where $A$, $B$, $C$ and $D$ are each $k$ by $k$ matrices. We then denote by $\hat{x}$ the following $2k+4$ by $2k+4$ matrix
\[ \hat{x} = \left( \begin{array}{cc} \tilde{A} & \tilde{B} \\
                               \tilde{C} & \tilde{D} \end{array} \right) \]
\begin{lemma} \label{lemma-6_4} If $\Span_{\mathbb{R}} \{ x_i: 1 \leq i \leq k \}$ is a Cartan subalgebra of $\mathfrak{so}(2k)$ 
which is orthogonal to $C_{2k}$, then $\Span_{\mathbb{R}} \{ \hat{x}_i, y, z; 1 \leq i \leq k \}$ is a Cartan subalgebra of 
$\mathfrak{so}(2k+4)$ which is orthogonal to $C_{2k+4}$, where
\begin{align*} y &= \left( \begin{array}{cc} 0_k & S \\
                                             -S & 0_k \end{array} \right) \\
               z &= \left( \begin{array}{cc} 0_k & T \\
                                             T & 0_k \end{array} \right)
\end{align*}
where
\begin{align*}
S &= 0_k \oplus \left( \begin{array}{cc} 0 & 1 \\
                                         1 & 0 \end{array} \right) \\
T &= 0_k \oplus \left( \begin{array}{cc} 0 & -1 \\
                                         1 & 0 \end{array} \right)
\end{align*}
\end{lemma}
\begin{proof} It suffices to prove that the $\hat{x}_i$s, $y$ and $z$ all mutually commute, which is easy to check.
\end{proof}
Hence, using lemmas \ref{lemma-6_3} and \ref{lemma-6_4}, it suffices for our purposes to construct a Cartan 
subalgebra of $\mathfrak{so}(4)$, respectively $\mathfrak{so}(6)$ which is orthogonal to $C_4$, respectively $C_6$. Then we can, 
using the constructions in these two lemmas, construct orthogonal Cartan subalgebras of $\mathfrak{so}(n)$ for any $n \geq 4$ (the case 
of $\mathfrak{so}(3)$ is trivial).

Let us consider first $\mathfrak{so}(4)$. Define
\begin{align*} x_1 &= \left( \begin{array}{cccc} 0 & 0 & 0 & 1 \\
                                                 0 & 0 & 1 & 0 \\
                                                 0 &-1 & 0 & 0 \\
                                                -1 & 0 & 0 & 0 \end{array} \right) \\
x_2 &= \left( \begin{array}{cccc} 0 & 0 & 0 & -1 \\
                                  0 & 0 & 1 & 0 \\
                                  0 &-1 & 0 & 0 \\
                                  1 & 0 & 0 & 0 \end{array} \right)
\end{align*}
It is easy to check that $x_1$ and $x_2$ span a Cartan subalgebra of $\mathfrak{so}(4)$ which is orthogonal to $C_4$.

We finally consider the case of $\mathfrak{so}(6)$. Define
\begin{align*} A &= \left( \begin{array}{ccc} 0 & -1 & 0 \\
                                               1 &  0 & 0 \\
                                               0 &  0 & 0 \end{array} \right) \\
B &= \left( \begin{array}{ccc} 0 & 0 & 0 \\
                                0 & 0 & 0 \\
                                1 & 0 & 0 \end{array} \right) \\
C &= \left( \begin{array}{ccc} 0 & 0 & 0 \\
                               0 & 0 &-1 \\
                               0 & 1 & 0 \end{array} \right)
\end{align*}
We then define $x_1 = A \oplus 0_3$, and
\[ x_2 = \left( \begin{array}{cc} 0_3 & B \\
                                 -B^T & 0_3 \end{array} \right) \]
as well as $x_3 = 0_3 \oplus C$.

Then a straightforward computation shows that $x_1$, $x_2$ and $x_3$ mutually commute, so that they span a Cartan subalgebra of 
$\mathfrak{so}(6)$, which is also orthogonal to $C_6$ (since each of $x_1$, $x_2$ and $x_3$ is orthogonal to $C_6$). This finishes our 
description of orthogonal Cartan subalgebras for $\mathfrak{so}(n)$, for $n \geq 4$, with the case of $n=3$ being trivial. 

\appendix
\section{Some consequences of Kostant's convexity theorem} \label{sectionA}

As noted in the introduction, we shall present, in this Appendix, a slight simplification of Karl-Hermann Neeb's proof of 
Goto's Theorem, additive version, using Kostant's Convexity Theorem \cite[p. 653]{H-M2}. His proof essentially passes through 
the proof of Theorem \ref{thm-3_2}.
 
\begin{theorem}[Kostant's convexity theorem] \label{thm-A_1} Let $L=\Lie(G)$ be the Lie algebra of a compact semisimple Lie group. 
Let $C=\Lie(T)$ be 
a maximal toral subalgebra of $L$, corresponding to a maximal torus $T \subseteq G$. Let $p: L \to C$ be the orthogonal projection of 
$L$ onto $C$ with respect to the negative of the Killing form on $L$. Let $x$ be an element of $C$, and let $W$ be the Weyl group of 
$(G,T)$ or $(L,C)$. Then $p(G.x) = \Conv(W.x)$. 
\end{theorem} 
This version of Kostant's Convexity Theorem follows very easily from the original version \cite[Theorem 8.2]{Kos} by taking 
the complexification $L(G)_{\mathbb{C}}$ of $L(G)$, viewed as a real semisimple Lie algebra. Then $L(G)_{\mathbb{C}}$ has a 
Cartan decomposition $\mathfrak{k} \oplus \mathfrak{p}$ where $\mathfrak{k} = L(G)$ and $\mathfrak{p} =  iL(G)$. So we apply 
Theorem 8.2 of \cite{Kos} to $L(G)_{\mathbb{C}}$ viewed as a real semisimple Lie algebra, with $\mathfrak{k}$ and 
$\mathfrak{p}$ as above.
\begin{lemma}\cite[lemma $2.2$]{Akh} \label{lemma-A_2} Let $V$ be a real vector space and $W \subset GL(V)$ a finite linear group acting 
without fixed vectors (other than $0$). Then the convex hull of any $W$-orbit contains $0 \in V$. \end{lemma}

For the reader's convenience, we explain the idea of its short proof in \cite{Akh} which is that the centroid $z$ of the points of any 
$W$-orbit in $V$ is fixed under the action of $W$. But $z=0$ is the only fixed point under $W$. Hence $z=0$ belongs to the 
convex hull of any $W$-orbit. 

Now we give another proof of Theorem \ref{thm-3_2}.
\begin{theorem} \label{thm-A_3} (= Theorem \ref{thm-3_2}) In the setting of theorem \ref{thm-A_1}, let $C^{\perp}$ be the orthogonal complement of $C$ (with respect 
to negative the Killing form on $L$). Then for every element $x$ of $L$, there exists an element $g \in G$ such that $g.x$ belongs to 
$C^{\perp}$.  
\end{theorem}
\begin{proof} Combine \ref{thm-A_1} and \ref{lemma-A_2}.
\end{proof}

\begin{theorem}[Goto's theorem, additive version] \label{thm-A_4} In the Lie algebra $L$ of a compact semisimple Lie group, every element $x$ is a 
bracket. That is, $x=[a,b]$ for some $a$, $b$ in $L$. Moreover, $a$ or $b$ can be chosen to be regular. 
\end{theorem}
\begin{proof} By \ref{thm-A_3}, there exists an element $g$ in $G$ such that $g.x$ belongs 
to $C^{\perp}$. But $C = Z(a)$ for some regular element $a$ of $L$. Moreover, $Z(a)^{\perp} = [a,L]$ by \ref{remark-5_1_2}. 
Hence $g.x$ belongs to $[a,L]$. Consequently, $x$ belongs to $[g^{-1}.a,L]$, and $g^{-1}.a$ is regular.
\end{proof}  

\section{Alternative proof of Cor. \ref{cor-1_3} by $SU(2)$-rotations} \label{sectionB}

In this section, we give a more direct proof of Cor. \ref{cor-1_3} without using 
Kostant's convexity theorem. The proof is based on $SU(2)$ rotations in the orthogonal ``root'' space decomposition 
of the Lie algebra $L=\Lie(G)$ where $G$ is a compact semisimple Lie group.

We equip $L$ with its natural inner product (negative the Cartan Killing form of $L$). Let $C$ be a maximal toral 
subalgebra of $L$. Let $\Sigma \subseteq C^*$ (the dual of $C$) be the root system of $G$ with respect to $C$ and we 
fix an ordering of the roots with $\Delta$ as the set of positive roots. 

Then under the adjoint action of $C$, $L$ has an orthogonal decomposition
\[ L = C \oplus \sum_{\alpha \in \Delta} L_{\alpha} \]

Let $a$ be a positive root, and consider $\Ker(a) = \{ x \in C: a(x) = 0 \}$, and its orthogonal space 
$\Ker(a)^{\perp}$ in $C$. Then we have
\[ L = \Ker(a) \oplus \Ker(a)^{\perp} \oplus \sum_{\alpha \in \Delta} L_{\alpha} \]
where $\Ker(a)^{\perp} \oplus L_a \simeq \frak{su}(2)$.

More specifically, each $L_{\alpha}$ has an orthogonal basis $\{u_{\alpha}, v_{\alpha} \}$ and $\Ker(\alpha)^{\perp}$ has a basis 
$\{h_{\alpha} \}$ such that
\begin{enumerate}
\item $[h, u_{\alpha}] = \alpha(h) v_{\alpha}$ for all $h$ in $C$,
\item $[h, v_{\alpha}] = -\alpha(h) u_{\alpha}$ for all $h$ in $C$, and $\{\alpha, u_{\alpha}, v_{\alpha} \}$,
\item $[u_{\alpha}, v_{\alpha}] = h_{\alpha}$ in $C$,
\item $\Ker(\alpha)^{\perp} \oplus L_{\alpha} = \Span_{\mathbb{R}}\{h_{\alpha}, u_{\alpha}, v_{\alpha} \} \simeq \mathfrak{su}(2)$.
\end{enumerate}
(cf. \cite{H-M1}, 6.48, 6.49], \cite{Knp}, p. 353, \cite{Zlr1}, p. 59)

Let $S_a = \Ker(a)^{\perp} \oplus L_a$, and let $m_a = \sum_{\alpha \in \Delta \setminus \{a\}}L_{\alpha}$. Then
\[ L = \Ker(a) \oplus S_a \oplus m_a \]
where $S_a \simeq \frak{su}(2)$.

\begin{lemma} \label{lemma-4d_1} (cf. \cite{Wld}, lemma 2.2) In the above setting where
\[ L = \Ker(a) \oplus S_a \oplus m_a \]
where $a \in \Delta$ is a positive root, let $x \in C$. Then there exists a $g$ in the Lie subgroup $G_a \simeq SU(2)$ 
corresponding to $S_a \simeq \mathfrak{su}(2)$, such that
\begin{enumerate}
\item $g.x$ has no $\Ker(a)^{\perp}$ component \label{part-a}
\item $x$ and $g.x$ have the same $Ker(a)$ components \label{part-b}
\item the $m_a$ components of $x$ and $g.x$ have the same norms \label{part-c}
\end{enumerate}
\end{lemma}
\begin{proof} Let $x'$ be the component of $x$ in $\Ker(a)^{\perp}$. So $x' \in S_a = \Ker(a)^{\perp} \oplus L_a \simeq \mathfrak{su}(2)$. 
Hence there exists a $g$ in $G_a \simeq SU(2)$ such that $g.x'$ has no $\Ker(a)^{\perp}$ component. We do this via rotations in 
$SU(2)$ (or the fact that all Cartan subalgebras of $\mathfrak{su}(2)$ are conjugate). Because we are dealing with $\Ker(a)$, 
the decomposition
\begin{equation} L = \Ker(a) \oplus S_a \oplus m_a \label{decomp} \end{equation}
is easily checked to be invariant under the action of $S_a = \Ker(a)^{\perp} \oplus L_a$. Hence $g.x$ has no $\Ker(a)^{\perp}$ 
component, which proves \ref{part-a}. Moreover, $[S_a, \Ker(a)] = [L_a, \Ker(a)] = 0$. Consequently, the group $G_a \simeq SU(2)$ fixes $\Ker(a)$. Hence 
$x$ and $g.x$ have the same $\Ker(a)$ components, thus proving \ref{part-b}.
Finally, since the action of $G_a$ preserves the decomposition \ref{decomp}, and every element of $G_a$ acts by orthogonal 
transformations (with respect to minus the Killing form), \ref{part-c} follows.
\end{proof}

\begin{theorem}[equiv. to Cor. \ref{cor-1_3}] \label{thm-4d_2} Let $L = \Lie(G)$ be the Lie algebra of a compact semisimple Lie group. Let $C$ be a maximal toral subalgebra 
of $L$, and let $x \in L$. Then there exists a $g \in G$, such that $g.x$ belongs to $C^{\perp}$ (in particular, the $G$-orbit of 
any $C^{\perp}$ is all of $L$). 
\end{theorem}
\begin{proof} The following is a proof by contradiction. Assume not. Then there is an $x \in L$, necessarily non-zero, such that 
$g.x$ is not in $C^{\perp}$, for all $g \in G$. But $G$ is compact, so that $G.x$ is also compact (being the continuous image of $G$), so 
there is a $y \in G.x$ with the property that its $C$-component has minimal norm, among all points in $G.x$. By our assumption, 
the norm of the $C$-component of $y$ is positive. In order to obtain a contradiction, we shall exhibit a $y' \in G.x$ whose 
$C$-component has norm which is strictly less than that of $y$. This will follow by an easy application of lemma \ref{lemma-4d_1}. 
Indeed, since $y$ is not in $C^{\perp}$, it follows that there is a positive root $a \in \Delta$ such that the $C$-component of 
$y$ is not in the kernel of $a$. Apply lemma \ref{lemma-4d_1} to $y$ and such a positive root $a$. Thus, there exists an 
$y' = g.y \in G.y = G.x$, for some $g \in G$, such that
\begin{enumerate}
\item $y'$ has no $\Ker(a)^{\perp}$ component
\item $y$ and $y'$ have the same $\Ker(a)$ components
\item the $m_a$-components of $y$ and $y'$ have the same norms
\end{enumerate}
Writing
\begin{align*} y &= y_1 + y_2 + y_3 \\
               y'& = y'_1 + y'_2 + y'_3 \end{align*}
where $y_1$, $y'_1$ are elements of $\Ker(a)^{\perp}$, $y_2$, $y'_2$ are in $\Ker(a)$, and $y_3$, $y'_3$ are elements of 
$C^{\perp}$, we have that $y'_1 = 0$ and that $y'_2 = y_2$. 
\[ \lvert y'_1 + y'_2 \rvert^2 = \lvert y_2 \rvert^2 < \lvert y_1 \rvert^2 + \lvert y_2 \rvert^2 = \lvert y_1 + y_2 \rvert^2  \]
The previous inequality is strict since $y_1$ is non-zero, by our choice of positive root $a$. Since $y_1+y_2$ is the $C$-component of $y$ (and similarly $y'_1+y'_2$ is the $C$-component of $y'$), this contradicts the 
property of $y$ having minimal $C$-component norm among all points in $G.x$, thus finishing the proof.
\end{proof}

\vspace{5mm}

\def\Dbar{\leavevmode\lower.6ex\hbox to 0pt{\hskip-.23ex \accent"16\hss}D}
\providecommand{\bysame}{\leavevmode\hbox to3em{\hrulefill}\thinspace}
\providecommand{\MR}{\relax\ifhmode\unskip\space\fi MR }
\providecommand{\MRhref}[2]{%
  \href{http://www.ams.org/mathscinet-getitem?mr=#1}{#2}
}
\providecommand{\href}[2]{#2}


\begin{thebibliography}{1}

\bibitem{Akh}
D. Akhiezer, \emph{On the commutator map for real semisimple Lie algebras}, Moscow Mathematical Journal, Vol. \textbf{15}, No. 4 (2015), pp. 609 - 613.


\bibitem{B-N}
G.~M. Bergman and N. Nahlus, \emph{Homomorphisms on infinite direct product algebras, especially Lie algebras}. 
J. Algebra \textbf{333} (2011), 67-104.

\bibitem{Bor1}
N. Bourbaki, \emph{Lie groups and Lie algebras. Chapters 1-3}. Translated from the French. Reprint of the 
1975 English translation. Elements of Mathematics (Berlin). Springer-Verlag, Berlin, 1989, xviii+450 pp.

\bibitem{Bor2}
N. Bourbaki, \emph{Lie groups and Lie algebras. Chapters 4-6}. Translated from the 1968 French original by 
Andrew Pressley. Elements of Mathematics (Berlin). Springer-Verlag, Berlin, 2002, xii+300 pp.

\bibitem{Bor3}
N. Bourbaki, \emph{Lie groups and Lie algebras. Chapters 7-9}. Translated from the 1975 and 1982 French originals 
by Andrew Pressley. Elements of Mathematics (Berlin). Springer-Verlag, Berlin, 2005, xii+434 pp.

\bibitem{Brn}
G. Brown, \emph{On commutators in a simple Lie algebra}. Proc. Amer. Math. Soc. \textbf{14} 1963, 763-767. 

\bibitem{A-M}
A. D'Andrea and A. Maffei, \emph{Small commutators in compact semisimple Lie groups and Lie algebras}, J. Lie theory \textbf{26} (2016), no. 1, 683-690.

\bibitem{D-T}
D.~Z. {\Dbar}okovi{\'c} and T.-Y. Tam, \emph{Some questions about semisimple Lie groups originating in matrix theory}. Canad. Math. 
Bull. \textbf{46} (2003), no. 3, 332-343. 

\bibitem{F-Z}
L.~A. Florit and W. Ziller, \emph{Topological obstructions to fatness}. Geom. Topol. \textbf{15} (2011), no.2, 891-925.

\bibitem{Got}
M. Got\^{o}, \emph{A theorem on compact semi-simple groups}. J. Math. Soc. Japan \textbf{1} (1949), 270-272.

\bibitem{H-N}
J. Hilgert and K.-H. Neeb, \emph{Structure and geometry of Lie groups}. Springer Monographs in Mathematics. Springer, 
New York, 2012. x+744 pp.

\bibitem{H-M1}
K.~H. Hofmann and S.~A. Morris, \emph{The Structure of Compact Groups. A primer for the student--a handbook for the expert}, 
Second revised and augmented edition, De Gruyter Studies in Mathematics, \textbf{25}, Walter de Gruyter \and Co., Berlin, 2006.

\bibitem{H-M2}
\bysame and \bysame, \emph{The Lie theory of connected pro-Lie groups}, a structure theory for pro-Lie algebras, pro-Lie groups, 
and connected locally compact groups, EMS tracts in Mathematics, \textbf{2}, European Mathematical Society (EMS), Z\"{u}rich, 2007.


\bibitem{Hum2}
J.~E. Humphreys, \emph{Reflection groups and Coxeter groups}. Cambridge Studies in Advanced Mathematics, \textbf{29}. 
Cambridge University Press, Cambridge, 1990. xii + 204 pp.

\bibitem{Knp}
A.~W. Knapp, \emph{Lie groups beyond an introduction}, Second edition, Progress in Mathematics, \textbf{140}, Birkh\"{a}user Boston, 
Inc., Boston, MA, 2002.

\bibitem{Kos}
B. Kostant, \emph{On convexity, the Weyl group and the Iwasawa decomposition}. Ann. Sci. {\'E}cole Norm. Sup. (4) \textbf{6} (1973), 
413-455 (1974).

\bibitem{K-S}
I. Kra and S. Simanca, \emph{On circulant matrices}. Notices Amer. Math. Soc. \textbf{59} (2012), no. 3, 368-377.

\bibitem{Ler}
E. Lerman, \emph{How fat is a fat bundle?} Lett. Math. Phys. \textbf{15} (1988), no.4, 335-339.



\bibitem{N}
N. Nahlus, \emph{On $1.5$ generators and $L = [L,a]+[L,b]$ in simple Lie algebras}, in preparation, tentative title.

\bibitem{Var}
V.~S. Varadarajan, \emph{Lie groups, Lie algebras, and their representations}. Prentice-Hall Series in Modern Analysis. 
Prentice-Hall, Inc., Englewood Cliffs, N.J., 1974, xiii+430 pp.

\bibitem{Wld}
N.~J. Wildberger, \emph{Diagonalization in compact Lie algebras and a new proof of a theorem of Kostant}. Proc. Amer. Math. Soc. 
\textbf{119} (1993), no. 2, 649-655.

\bibitem{Wei}
A. Weinstein, \emph{Fat bundles and symplectic manifolds}, Adv. Math. \textbf{37} (1980), 239--250

\bibitem{Zlr1}
W. Ziller, \emph{Lie groups, representation theory and symmetric spaces}, notes from courses taught Fall 2010 at UPenn, and 2012 at 
IMPA, \url{http://www.math.upenn.edu/~wziller/}

\end{thebibliography}
\end{document}